\documentclass{amsart}
\usepackage{bm}
\usepackage{cases}
\usepackage[all]{xy}
\usepackage{amsmath}
\usepackage{amssymb}
\usepackage{amscd}
\usepackage{amsthm}
\usepackage[dvips]{graphicx}
\usepackage{color}

\theoremstyle{definition}
\newtheorem{defn}{Definition}[section]

\newtheorem{rem}[defn]{Remark}
\theoremstyle{plain}
\newtheorem{thm}[defn]{Theorem}

\newtheorem{prob}[defn]{Problem}
\numberwithin{equation}{section}
\allowdisplaybreaks
\title[Small generating set for the twist subgroup]{A small generating set for the twist subgroup of the mapping class group of a non-orientable surface by Dehn twists}

\author[G.~Omori]{Genki Omori}
\address{
(Genki Omori)
Department of Mathematics,
Tokyo Institute of Technology,
Oh-okayama, Meguro, Tokyo 152-8551, Japan
}

\email{omori.g.aa@m.titech.ac.jp}
\date{\today}
\begin{document}
\maketitle
\begin{abstract}
We give a small generating set for the twist subgroup of the mapping class group of a non-orientable surface by Dehn twists. The difference between the number of the generators and a lower bound of numbers of generators for the twist subgroup by Dehn twists is one. The lower bounds is obtained from an argument of Hirose~\cite{Hirose}. 
\end{abstract}

\section{Introduction}

Let $\Sigma _{g,n}$ be a compact connected oriented surface of genus $g\geq 0$ with $n\geq 0$ boundary components and we put $\Sigma _g:=\Sigma _{g,0}$. The {\it mapping class group} $\mathcal{M}(\Sigma _{g,n})$ of $\Sigma _{g,n}$ is the group of isotopy classes of orientation preserving self-diffeomorphisms on $\Sigma _{g,n}$ fixing the boundary pointwise. Dehn~\cite{Dehn} proved that $\mathcal{M}(\Sigma _g)$ is generated by $2g(g-1)$ Dehn twists. The generating set includes Dehn twists along separating simple closed curves. Mumford~\cite{Mumford} showed that $\mathcal{M}(\Sigma _g)$ is generated by Dehn twists along non-separating simple closed curves, and Lickorish~\cite{Lickorish2} gave a finite generating set for $\mathcal{M}(\Sigma _g)$ by $3g-1$ Dehn twists along non-separating simple closed curves. By an argument in Proof of Theorem~4.13 in \cite{Farb-Margalit}, $\mathcal{M}(\Sigma _{g,1})$ is also generated by $3g-1$ Dehn twists along non-separating simple closed curves. After that, Humphries~\cite{Humphries} proved that $\mathcal{M}(\Sigma _{g,n})$ is generated by a subset of Lickorish's generating set whose cardinality is $2g+1$ for $g\geq 2$ and $n\in \{ 0,1\}$, and he also proved that the generating set is minimal in generating sets for $\mathcal{M}(\Sigma _{g,n})$ by Dehn twists.

Let $N_{g,n}$ be a compact connected non-orientable surface of genus $g\geq 1$ with $n\geq 0$ boundary components. The surface $N_g:=N_{g,0}$ is a connected sum of $g$ real projective planes. The mapping class group $\mathcal{M}(N_{g,n})$ of $N_{g,n}$ is the group of isotopy classes of self-diffeomorphisms on $N_{g,n}$ fixing the boundary pointwise. For $n\in \{ 0,1\}$, $\mathcal{M}(N_{1,n})$ is the trivial group (see \cite[Theorem~3.4]{Epstein}). 
For $g\geq 2$, Lickorish proved that $\mathcal{M}(N_g)$ is not generated by Dehn twists in \cite{Lickorish1}, and $\mathcal{M}(N_{g,n})$ is generated by Dehn twists and a ``Y-homeomorphism" in \cite{Lickorish1, Lickorish3}.
The Y-homeomorphism is introduced by Lickorish in \cite{Lickorish1}.
Lickorish~\cite{Lickorish1} also showed that $\mathcal{M}(N_2)$ is generated by a Dehn twist and a Y-homeomorphism. In generally, Chilling worth~\cite{Chillingworth} gave a finite generating set for $\mathcal{M}(N_g)$ which consists of $\frac{3g-5}{2}$ (resp. $\frac{3g-6}{2}$) Dehn twists and a Y-homeomorphism for odd (resp. even) $g$. After that, Szepietowski~\cite{Szepietowski} proved that $\mathcal{M}(N_g)$ is generated by a subset of Chillingworth's generating set which consists of $g$ Dehn twists and a Y-homeomorphism, and Hirose~\cite{Hirose} showed that the generating set is minimal in generating sets for $\mathcal{M}(N_g)$ by Dehn twists and Y-homeomorphisms. By Stukow's finite presentation for $\mathcal{M}(N_{g,1})$ in \cite{Stukow1} and an argument in \cite{Hirose} (see Remark~\ref{remark}), $\mathcal{M}(N_{g,1})$ also has a minimal generating set by Dehn twists and Y-homeomorphisms which consists of $g$ Dehn twists and a Y-homeomorphism.

The {\it twist subgroup} $\mathcal{T}(N_{g,n})$ {\it of} $\mathcal{M}(N_{g,n})$ is the subgroup of $\mathcal{M}(N_{g,n})$ which is generated by all Dehn twists. $\mathcal{T}(N_{g,n})$ is an index 2 subgroup of $\mathcal{M}(N_{g,n})$ (see \cite{Lickorish3} and \cite[Corollary~6.4]{Stukow0}). In particular, $\mathcal{T}(N_{g,n})$ is finitely generated. Chillingworth~\cite{Chillingworth} showed that $\mathcal{T}(N_g)$ is generated by a Dehn twist for $g=2$, two Dehn twists for $g=3$, $\frac{3g-1}{2}$ Dehn twists for the other odd $g$ and $\frac{3g}{2}$ Dehn twists for the other even $g$. By an argument as in \cite{Humphries}, we can reduce the number of Chillingworth's generators to $g+2$ for odd $g>3$ and $g+3$ for even $g>3$. For $n\in \{ 0,1\}$, Stukow~\cite{Stukow2} gave a finite presentation for $\mathcal{T}(N_{g,n})$ whose generators are $g+2$ Dehn twists essentially by relations of the presentation (see Proof of Theorem~\ref{mainthm}). 

In this paper we proved that $\mathcal{T}(N_{g,n})$ is generated by $g+1$ Dehn twists for $g\geq 4$ (Theorem~\ref{mainthm}). The generating set is a proper subset of the generating set of Stukow's finite presentation in \cite{Stukow2}. By applying Hirose's argument in \cite{Hirose}, the difference between the number of the generators in Theorem~\ref{mainthm} and a lower bound of numbers of generators for $\mathcal{T}(N_{g,n})$ by Dehn twists is one (see Remark~\ref{remark}). The author does not know whether the generating set for $\mathcal{T}(N_{g,n})$ in Theorem~\ref{mainthm} is minimal in generating sets for $\mathcal{T}(N_{g,n})$ by Dehn twists or not.

\section{Preliminaries}\label{Preliminaries}

For a two-sided simple closed curve $c$ on $N_{g,n}$, we take an orientation of the regular neighborhood of $c$ in $N_{g,n}$. Then we denote by $t_c$ the right-handed Dehn twist along $c$ with respect to the orientation. In particular, for a given explicit two-sided simple closed curve, an arrow on a side of the simple closed curve indicates the direction of the Dehn twist (see Figure~\ref{dehntwist}).

\begin{figure}[h]
\includegraphics[scale=0.65]{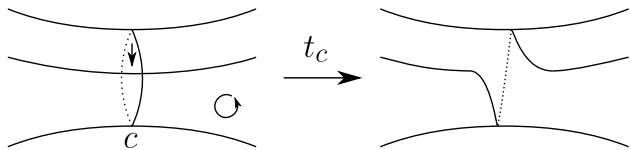}
\caption{The right-handed Dehn twist $t_c$ along a two-sided simple closed curve $c$ on $N_{g,n}$.}\label{dehntwist}
\end{figure}

Let $e_i:D \hookrightarrow \Sigma _0$ for $i=1$, $2, \dots $, $g+1$ be smooth embeddings of the unit disk $D$ to a 2-sphere $\Sigma _0$ such that $D_i:=e_i(D)$ and $D_j$ are disjoint for distinct $1\leq i,j\leq g+1$. Then we take a model of $N_g$ (resp. $N_{g,1}$) as the surface obtained from $\Sigma _0-{\rm int}(D_1\sqcup \cdots \sqcup D_g)$ (resp. $\Sigma _0-{\rm int}(D_1\sqcup \cdots \sqcup D_{g+1})$) by identifying antipodal points of the boundary components of $D_1,\dots ,D_g$ and we describe the identification of $\partial D_i$ by the x-mark as in Figure~\ref{sccs}. 

For $n\in \{ 0,1\}$, we denote by $\alpha _1$, $\dots $, $\alpha _{g-1}$ and $\beta $ two-sided simple closed curves on $N_{g,n}$ as in Figure~\ref{sccs}, and denote by $\gamma $, $\varepsilon $, $\zeta $ and $\psi $ two-sided simple closed curves on $N_{g,n}$ as in Figure~\ref{sccs2}, respectively. Then we define $a_i:=t_{\alpha _i}$ $(i=1,\dots ,g-1)$, $b:=t_\beta $, $e:=t_\varepsilon $, $f:=t_\zeta $, $y^2:=t_\psi $ and $c:=t_\gamma $.

\begin{figure}[h]
\includegraphics[scale=0.7]{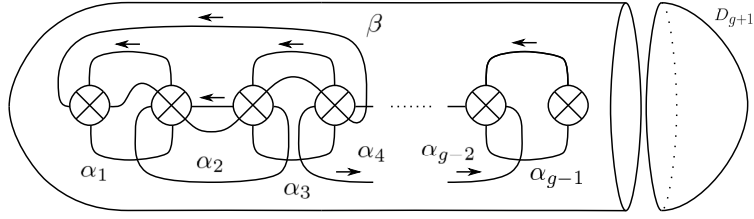}
\caption{Simple closed curves $\alpha _1$, $\dots $, $\alpha _{g-1}$ and $\beta $ on $N_{g,n}$.}\label{sccs}
\end{figure}

\begin{figure}[h]
\includegraphics[scale=0.6]{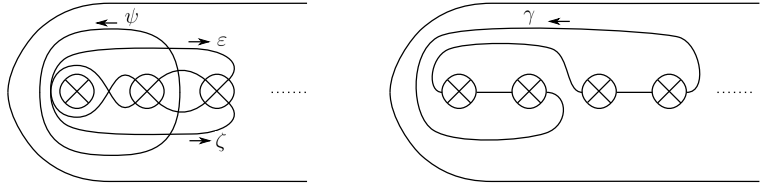}
\caption{Simple closed curves $\psi $, $\varepsilon $, $\zeta $ and $\gamma $ on $N_{g,n}$.}\label{sccs2}
\end{figure}

\section{Main result}\label{section_mainthm}

The main theorem in this paper is as follows.
\begin{thm}\label{mainthm}
For $g\geq 4$ and $n\in \{ 0,1\}$, $\mathcal{T}(N_{g,n})$ is generated by $a_1$, $\dots $, $a_{g-1}$, $b$ and $e$. In particular, $\mathcal{T}(N_{g,n})$ is generated by $g+1$ Dehn twists along non-separating simple closed curves.
\end{thm}

\begin{proof}
Assume $g\geq 4$ and $n\in \{ 0,1\}$. Stukow's presentation for $\mathcal{T}(N_{g,n})$ in \cite{Stukow2} has the following generating set: 
\begin{itemize}
\item $X:=\{a_1, \dots , a_{g-1}, b, e, f, y^2, c\}$ \hspace{0.2cm} for odd $g$ and $n=1$ or $g=4$ and $n=1$,
\item $X^\prime :=X\cup \{ b_0,b_1, \dots , b_{\frac{g-2}{2}}, \bar{b}_{\frac{g-6}{2}}, \bar{b}_{\frac{g-4}{2}}, \bar{b}_{\frac{g-2}{2}}\}$ \hspace{0.2cm} for even $g\geq 6$ and $n=1$,
\item $X\cup\{ \rho \}$ \hspace{0.2cm} for odd $g$ and $n=0$,
\item $X \cup \{ \bar{\rho }\}$ \hspace{0.2cm} for $g=4$ and $n=0$,
\item $X^\prime \cup \{ \bar{\rho }\}$ \hspace{0.2cm} for even $g\geq 6$ and $n=0$.
\end{itemize}
$b_0$, $b_1$, $\dots $, $b_{\frac{g-2}{2}}$, $\bar{b}_{\frac{g-6}{2}}$, $\bar{b}_{\frac{g-4}{2}}$, $\bar{b}_{\frac{g-2}{2}}$, $\rho $ and $\bar{\rho }$ are products of elements in $X$ by the relations~(A7), (A8), ($\overline{\text{A7a}}$)-($\overline{\text{A8b}}$), (C1a) and ($\overline{\text{C4}}$) in Theorem~2.1, Theorem~2.2, Theorem~3.1 and Theorem~3.2 of \cite{Stukow2}. Thus $\mathcal{T}(N_{g,n})$ is generated by $X$. 
By the relation~($\overline{\text{B2}_1}$) in Theorem~3.1 of \cite{Stukow2}, $y^2$ is a product of elements in $X-\{ y^2\}$, and by the relation~($\overline{\text{B6}_1}$) in Theorem~3.1 of \cite{Stukow2}, $c$ is a product of $a_1$, $\dots $, $a_{g-1}$, $b$, $e$ and $f$.

Finally, we can check $a_3^{-1}a_2^{-1}ba_1^{-1}a_2^{-1}a_3^{-1}(\varepsilon )=\zeta $ and the orientation of the regular neighborhood of $a_3^{-1}a_2^{-1}ba_1^{-1}a_2^{-1}a_3^{-1}(\varepsilon )$ is different from one of $f$ as in Figure~\ref{proof}. Hence we have $f=(a_3^{-1}a_2^{-1}ba_1^{-1}a_2^{-1}a_3^{-1})e^{-1}(a_3^{-1}a_2^{-1}ba_1^{-1}a_2^{-1}a_3^{-1})^{-1}$. Therefore $\mathcal{T}(N_{g,n})$ is generated by $a_1$, $\dots $, $a_{g-1}$, $b$ and $e$, and we have completed the proof of Theorem~\ref{mainthm}.
\end{proof}

\begin{figure}[h]
\includegraphics[scale=0.6]{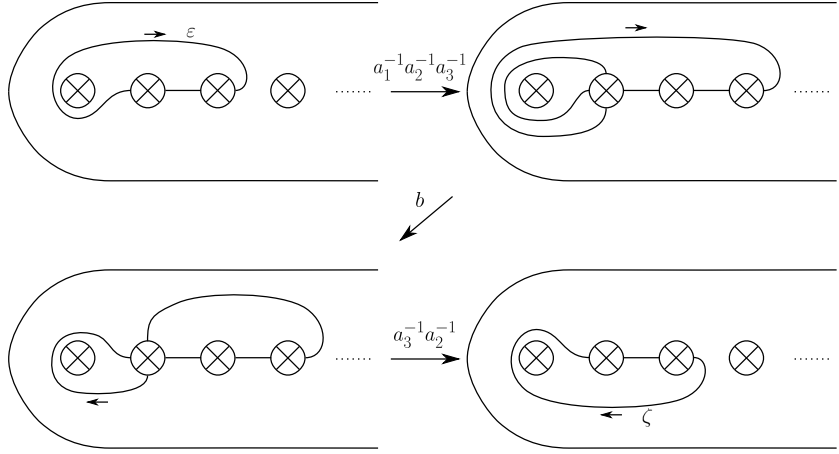}
\caption{Proving that $a_3^{-1}a_2^{-1}ba_1^{-1}a_2^{-1}a_3^{-1}(\varepsilon )=\zeta $.}\label{proof}
\end{figure}

\begin{rem}\label{remark0}
The regular neighborhood $N$ of the union of $\alpha _1$, $\dots $, $\alpha _{g-1}$ is an orientable subsurface of $N_{g,n}$ and $\{ a_1, \dots , a_{g-1}, b\}$ is the minimal generating set for $\mathcal{M}(N)$ by Dehn twists which is given by Humphries~\cite{Humphries}. Remark that $N_{g,n}-{\rm int}N$ is not a disjoint union of disks, and an element of the subgroup of $\mathcal{T}(N_{g,n})$ which is generated by $a_1$, $\dots $, $a_{g-1}$ and $b$ is represented by a diffeomorphism of $N_{g,n}$ whose restriction to $N_{g,n}-{\rm int}N$ is the identity map. However, $e$ does not fix $N_{g,n}-{\rm int}N$ up to ambient isotopies of $N_{g,n}$. Hence $\mathcal{T}(N_{g,n})$ is not generated by $a_1$, $\dots $, $a_{g-1}$ and $b$. Define $X_0:=\{ \alpha _1, \dots , \alpha _{g-1}, b, \varepsilon \}$. For $x_0\in \{ \alpha _4, \dots , \alpha _{g-1}, \varepsilon \}$, the complement $N_{g,n}-\displaystyle \bigcup _{x\in X_0\setminus \{ x_0\}}x$ has a non-disk component. Thus $\mathcal{T}(N_{g,n})$ is not also generated by $X_0-\{ x_0\}$ for $x_0\in \{ \alpha _4, \dots , \alpha _{g-1}, \varepsilon \}$.  
\end{rem}

\begin{rem}\label{remark}
We can apply Hirose's argument in \cite{Hirose} to $\mathcal{M}(N_{g,1})$ and $\mathcal{T}(N_{g,n})$ for $g\geq 4$ and $n\in \{ 0,1\}$. However, we should note that he take $\phi _j\in \mathcal{M}(N_{g,n})$ such that $\phi _j(c_1)=\gamma _j$ in the proof of Lemma~6 in \cite{Hirose}. To apply Hirose's argument in \cite{Hirose} to $\mathcal{T}(N_{g,n})$, we must take such $\phi _j$ as an element of $\mathcal{T}(N_{g,n})$. By using Lemma~7.2 in \cite{Stukow0}, we can take $\phi _j$ as an element of $\mathcal{T}(N_{g,n})$. Therefore the minimum number of generators for $\mathcal{T}(N_{g,n})$ by Dehn twists is at least $g$ for $g\geq 4$ and $n\in \{ 0,1\}$, and the difference between the number of the generators for $\mathcal{T}(N_{g,n})$ in Theorem~\ref{mainthm} and the lower bound of numbers of generators for $\mathcal{T}(N_{g,n})$ by Dehn twists is one.
\end{rem}

Finally we raise the following problem.

\begin{prob}
Which of $g$ and $g+1$ is the minimum number of generators for $\mathcal{T}(N_{g,n})$ by Dehn twists when $g\geq 4$ and $n\in \{ 0,1\}$? 
\end{prob}

\par
{\bf Acknowledgements: } The author would like to express his gratitude to Hisaaki Endo, for his encouragement and helpful advices. The author also wish to thank Susumu Hirose for his comments and helpful advices. The author was supported by JSPS KAKENHI Grant number 15J10066.

\end{document}